\documentclass[a4paper,twoside,final]{amsart}      
\usepackage{graphicx}
\usepackage{latexsym,amsmath,amsfonts,amssymb,amsaddr}
\usepackage[dvipsnames]{xcolor}
\usepackage[utf8]{inputenc}
\usepackage{layout,verbatim,setspace,enumerate}
\usepackage{amsthm,authblk,cite}
\newcommand{\vG}{\varGamma}
\newcommand{\ve}{\varepsilon}

\newcommand{\N}{\mathbb{N}}
\newcommand{\R}{\mathbb{R}}
\newcommand{\mcL}{{\mathcal L}}
\newcommand{\mcS}{{\mathcal S}}
\newcommand{\mcI}{{\mathcal I}}

\newcommand{\mcP}{{\mathcal P}}
\newcommand{\mcH}{\mathcal H}

\newcommand{\emp}{\emptyset}

\newcommand{\bs}{\boldsymbol}
\newcommand{\vPh}{\varPhi}
\newcommand{\mcM}{\mathcal M}

\newtheorem{thm}{Theorem}[section]
\numberwithin{thm}{section}
\newtheorem{deff}[thm]{Definition}
\newtheorem{lem}[thm]{Lemma}

\newtheorem{rem}[thm]{Remark}
\newtheorem{prop}[thm]{Proposition}
\newtheorem{cor}[thm]{Corollary}

\begin{document}
\title[Relations among Gauge and Pettis integrals for ...]{ Relations among Gauge and Pettis integrals for $cwk(X)$-valued multifunctions\\
\vskip1cm \small
D. Candeloro,  L. Di Piazza,  K. Musia{\l},    
A.R. Sambucini
}
\author[]{D. Candeloro;  L. Di Piazza;  K. Musia{\l};    
A.R. Sambucini}
\thanks{This research was partially supported by Grant  Prot. N.  U2016/0000807 of GNAMPA -- INDAM (Italy), by University of Perugia --
Dept. of Mathematics and Computer Sciences and by University of Palermo.
These results were obtained during the visit of   the third
author to the Dept. of Mathematics and Computer Sciences of the  University  of Perugia (Italy) as a visiting professor.
}


\begin{abstract}
The aim of this paper is to study   relationships
among ``gauge integrals''  (Henstock,  Mc Shane, Birkhoff)
and Pettis integral of multifunctions whose values are  weakly compact and convex subsets  of a general
Banach space, not necessarily separable. For this purpose we prove the existence of variationally Henstock integrable selections for variationally Henstock integrable multifunctions. Using this and other known results concerning the existence of selections integrable in the same sense as the corresponding multifunctions, we  obtain three decomposition
theorems (Theorem \ref{t4},  Theorem \ref{t4a} and Theorem \ref{decofinal}).  As  applications of such  decompositions,  we deduce characterizations of Henstock (Theorem \ref{t3}) and $\mcH$
(Theorem \ref{t3a}) integrable multifunctions, together with an
extension of  a well-known theorem of Fremlin  \cite[Theorem 8]{f1994a}.
\end{abstract}
 \maketitle

\section{Introduction}
A large amount of work about measurable and integrable multifunctions was done in the
last decades.  Some pioneering and highly influential ideas and notions around the matter
were inspired by problems arising in Control Theory and Mathematical Economics. But
the topic is  interesting also from the point of view of measure and integration theory, as showed in
the papers \cite{BCS,CASCALES2,ckr1,dm,dm2,dpp,mu8,m1,m2,mu4,BS2011,sbk,cs2014,cs2015,m,mp}.
 In particular, comparison of different generalizations of  Lebesgue integral is, in our opinion, one of the
milestones of the modern theory of integration. Inspired by  \cite{Fvol4, dm,ncm,ckr1,ckr2009, cdpms2016,nara,cdpms2016a}, we continue in this paper
 the study on this subject and we examine  relationship
among ``gauge integrals''  (Henstock,  Mc Shane, Birkhoff)
and Pettis integral of multifunctions whose values are  weakly compact and convex subsets  of a general
Banach space, not necessarily separable.

 The name  ``gauge integrals'' refers to integrals
 defined through  partitions controlled by a positive function, traditionally named  gauge. J.Kurzweil in 1957,  and then  R. Henstock in 1963, were the first who introduced a  definition of a gauge integral  for real valued functions, called now the  Henstock--Kurzweil integral. Its generalization to vector valued functions or to multivalued functions is called in the literature  the Henstock integral. In the family of the  gauge integrals there is also the  McShane integral and   the versions of the Henstock and the McShane integrals when  only measurable gauges are allowed ($\mathcal{H}$ and $\mathcal{M}$ integrals, respectively), and the variational Henstock and the variational McShane integrals.
Moreover  according to \cite{sol} and
\cite[Remark 1]{nara},  the Birkhoff integral is a gauge integral too and it turns out to be  equivalent to the $\mathcal{M}$ integral.

The main results of the paper  are
the existence of variationally Henstock integrable selections (Theorem \ref{T4.1}), which solves
the problem of the existence of  variationally Henstock integrable
selection for a
$cwk(X)$-valued variationally Henstock integrable multifunction
( \cite[Question 3.11]{cdpms2016}) and
 three decomposition
theorems (Theorem \ref{t4},  Theorem \ref{t4a} and Theorem \ref{decofinal}). The first one says that each  Henstock integrable multifunction
is the sum of a McShane  integrable multifunction and a Henstock integrable
function. The second one describes each $\mcH$-integrable multifunction
as the sum of a Birkhoff  integrable multifunction and  an $\mathcal{H}$-integrable
function, and the third one proves that each variationally Henstock integrable multifunction is the sum of a variationally Henstock integrable selection of the multifunction  and a   Birkhoff integrable multifunction  that is also  variationally Henstock integrable.
 As  applications of such  decomposition results,  characterizations of Henstock (Theorem \ref{t3}) and $\mcH$
(Theorem \ref{t3a}) integrable multifunctions are presented as extensions of the result given by
 Fremlin, in the remarkable paper \cite[Theorem 8]{f1994a}, and of more recent results given in
 \cite{dm,cdpms2016}.
 Finally we want to point out  that  in order to obtain the decomposition theorems and also the extension of the Fremlin result is not enough simply  to apply the embedding  theorem of R{\aa}dstr\"{o}m, but more sophisticated techniques are required.
\section{Preliminary facts}\label{two}
Let $[0,1] \subset \mathbb{R}$ be endowed
with the usual topology and  Lebesgue measure $\lambda$.
 The family of all Lebesgue measurable subsets of $[0,1]$ is denoted by $\mathcal{L}$, while
 $\mcI$
is the collection of all closed subintervals of $[0,1]$. If $I\in  \mcI$
then its Lebesgue  measure will be denoted by $|I|$.
\\
 A  finite  partition ${\mathcal P}$ {\it in} $[0,1]$ is a  collection $\{(I{}_1,t{}_1),$ $ \dots,(I{}_m,t{}_m) \}$,
  where $I{}_1,\dots,I{}_m$ are nonoverlapping  (i.e. the intersection  of two intervals is  at most
a singleton) closed
 subintervals of $[0,1]$, $t{}_i$ is a point of $[0,1]$, $i=1,\dots, m$.\
 If $\cup^m_{i=1} I{}_i=[0,1]$, then  ${\mathcal P}$ is {\it a partition of} $[0,1]$.\\
 If   $t_i \in I{}_i$, $i=1,\dots,m$,  we say that $\mcP$ is a {\it Perron partition of}
  $[0,1]$.\\
 A countable  partition $(A_n)_n$  of $[0,1]$ in $\mathcal{L}$ is a collection of pairwise disjoint $\mathcal{L}$-measurable sets such that $\cup_n A_n = [0,1]$; we admit empty sets.\\
 A {\it gauge} on $[0,1]$ is any
strictly  positive map on $[0,1]$.
Given a gauge $\delta$
  we say that a partition $\{(I{}_1,t{}_1), \dots,(I{}_m,t{}_m)\}$ is $\delta$-{\it fine} if
  $I{}_i\subset(t{}_i-\delta(t{}_i),t{}_i+\delta(t{}_i))$, $i=1,\dots,m$.
$\Pi_{\delta}$ and $\Pi_{\delta}^P$ are  the families of $\delta$-fine partitions, and $\delta$-fine Perron partitions of $[0,1]$, respectively.
\\
$X$ is an arbitrary  Banach space with its dual $X{}^*$.
  The closed unit ball of $X{}^*$ is denoted by
  $B_{X{}^*}$.
As usual $cwk(X)$ denotes the family of all non-empty  convex weakly compact subsets of $X$;
on this hyperspace the usual
   Minkowski addition  and the multiplication by positive scalars are considered, together with    the  Hausdorff distance
$d{}_H$. Moreover, $\|A\|:=\sup\{\|x\|\colon x\in {A}\}$. 
The {\it support
  function } $s: X^* \times cwk(X) \to \mathbb{R}$ is defined by $s(x{}^*, C)
:= \sup \{ \langle x{}^*,x \rangle : \ x
  \in  C\}$.

\begin{deff} \rm
  A map $\vG:[0,1]\to cwk(X)$ is called a {\it multifunction}.
$\vG$ is {\it simple} if  there exists a finite collection $\{A_1,..., A_p\}$ of measurable pairwise disjoint subsets of $[0,1]$  such that $\vG$ is constant on each $A_j$.\\
A map $\vG:\mcI\to cwk(X)$ is called an {\it interval multifunction}.
 A multifunction $\vG:[0,1]\to cwk(X)$ is said to be {\it scalarly measurable} if for every $ x{}^* \in  X{}^*$, the map
  $s(x{}^*,\vG(\cdot))$ is measurable.\\
 $\vG$ is said to be {\it Bochner measurable}  if there exists a sequence of simple multifunctions $\vG_n: [0,1] \to cwk(X)$ such that
$\lim_{n\rightarrow \infty}d_H(\vG_n(t),\vG(t))=0$ for almost all $t \in [0,1]$.\\
It is well known that Bochner
measurability of a $cwk(X)$-valued multifunction yields its  scalar
measurability. The reverse implication in general fails,  even if $X$ is separable (see \cite[ p.
295  and Example 3.8]{cdpms2016} ).
 \\

 If a multifunction is a function, then we use the traditional name of strong measurability instead of Bochner measurability.\\

  A function $f:[0,1]\to X$ is called a {\it selection of} $\vG$ if $f(t) \in \vG(t)$,   for every $t\in  [0,1]$.
\end{deff}
    \begin{deff} \rm
    A multifunction $\vG:[0,1]\to cwk(X)$ is said to be {\it Birkhoff
   integrable} on $[0,1]$,
   if there exists a set
$\vPh{}_{\vG}([0,1]) \in cwk(X)$
 with the following property: for every $\varepsilon > 0$  there is a countable
partition $\mcP{}_0$  of $[0,1]$  in $\mathcal{L}$ such that for every countable partition $\mcP = (A{}_n){}_n$
 of $[0,1]$  in $\mathcal{L}$
finer than $\mcP_0$ and any choice
 $T = \{t_n: t_n \in  A_n\,,n\in\N\}$, the series
$\sum_n\lambda(A{}_n) \vG(t{}_n)$
 is unconditionally
convergent (in the sense of the Hausdorff metric) and
\begin{eqnarray}\label{e14-a}
d{}_H \biggl(\vPh{}_{\vG}([0,1]),\sum_n \vG(t{}_n) \lambda(A{}_n)\biggr)<\ve\,.
\end{eqnarray}
(see for example  \cite[Proposition 2.6]{CASCALES2}).
\end{deff}

    \begin{deff} \rm
 A multifunction $\vG:[0,1]\to cwk(X)$ is said to be {\it Henstock} (resp. {\it McShane})
  {\it  integrable} on $[0,1]$,  if there exists   $\vPh{}_{\vG}([0,1]) \in cwk(X)$
    with the property that for every $\varepsilon > 0$ there exists a gauge $\delta$ on $[0,1]$
such that for each $\{(I{}_1,t{}_1), \dots,(I{}_p,t{}_p)\} \in \Pi_{\delta}^P$
(resp. $\in \Pi_{\delta}$)
we have
\begin{eqnarray}\label{e14}
d{}_H \biggl(\vPh{}_{\vG}([0,1]),\sum_{i=1}^p\vG(t{}_i)|I{}_i|\biggr)<\ve\,.
\end{eqnarray}
$\vG$ is said to be {\it Henstock} (resp. {McShane})
  {\it  integrable} on  $I \in \mathcal{I}$ ($E\in \mcL$) if
$\vG 1_I$
($\vG 1{}_E$) is  integrable on $[0,1]$
in the corresponding sense.
\end{deff}

  In case the multifunction is a single valued function and $X$ is the real line, the corresponding integral is called  {\it  Henstock--Kurzweil integral} (or  HK-{\it  integral}) and it is denoted by the symbol $(HK)\int_I$\;.

\begin{rem}\label{HH} \rm
If the gauges above considered are taken to be measurable, then we
speak of $\mathcal H$ (resp. $\mathcal M$)-integrability on $[0,1]$.

Given $\vG:[0,1]\to cwk(X)$, it is known  that the property of integrability is inherited on every $I\in \mcI$  if $\vG$ is
Henstock ($\mathcal H$)  integrable on $[0,1]$, while  the same is true for every
$E\in\mcL$ when $\vG$ is McShane ($\mathcal M$) integrable on $[0,1]$
 (see e.g. \cite{dm}).

 As pointed out before,
in case of single valued functions,
according to \cite{sol} and
\cite[Remark 1]{nara}, $\mathcal M$-integrability is equivalent to the Birkhoff integrability.
\end{rem}

\begin {deff} \rm
 A multifunction $\vG : [0; 1] \to cwk(X)$ is said to be {\it Henstock-Kurzweil-Pettis
integrable} (or HKP-integrable) on $[0,1]$ if for every $x^{*} \in X^{*}$ the map $s(x{}^*,\vG(\cdot))$
is HK-integrable and for each $I \in \mathcal{I}$ there exists a set $W_I \in cwk(X)$ such that
$s(x^*,W_I)  = (HK) \int_I s(x{}^*,\vG)$, for every $x^{*} \in X^{*}$. The set $W_I$ is called the Henstock-Kurzweil-Pettis integral of $\vG$ over $I$ and we set $W_I := (HKP)  \int_I \vG $.
\end{deff}

In the previous definition,
if  HK-integral is replaced by Lebesgue integral and intervals by Lebesgue measurable sets, then we get the definition of the Pettis integral.

 For more detailed properties of
the integrals involved  and for all that is unexplained in this paper
we refer   to \cite{mu,mu3,mu4,mu8,ckr1,dm,dm2, Gor}.

\begin{deff} \rm An interval multifunction  $\Phi:{\mathcal I} \rightarrow cwk(X)$ is said to be   \textit{finitely additive}, if $\Phi(I{}_1 \cup I{}_2)=\Phi(I{}_1) + \Phi(I{}_2)$ for every non-overlapping intervals $I{}_1, I{}_2 \in \mathcal{I}$ such that $I{}_1 \cup I{}_2 \in {\mathcal I}$.
In this case $\Phi$ is said to be an {\em interval multimeasure}.
\\
A map $M:\mathcal{L} \rightarrow cwk(X)$ is said to be a {\em  multimeasure} if for every $x{}^*\in X{}^*$, the map
$\mathcal{L}  \ni A\mapsto s(x{}^*,M(A))$ is a real valued measure
(cf. \cite[Theorem 8.4.10]{hp}).
\\
$M:\mathcal{L} \rightarrow cwk(X)$ is said to be a $d_H$\textit{-multimeasure} if for every sequence $(A_n)_{n\geq 1}$
in
$\mathcal{L}$ of pairwise disjoint sets with $A=\bigcup_{n\geq 1}A_n$, we have
$$
d_H\biggl(M(A),\sum_{k=1}^n M(A_k)\biggr)\rightarrow 0\, \quad
\mbox{ as }
n\rightarrow +\infty.$$
\\
A multimeasure $M:\mathcal{L} \rightarrow cwk(X)$ is  said to be  $\lambda$\textit{-continuous} and we write $M\ll\lambda$, if $M(A)=\{0\}$
for every $A \in \mathcal{L}$ such that
$\lambda(A)=0$.
\end{deff}
\begin{rem} \rm
 It is well known that
$M$ is a $d{}_H$-multimeasure if and only if it is a multimeasure (cf. \cite[Theorem 8.4.10]{hp}).
 Observe moreover that this is a multivalued analogue of Orlicz-Pettis Theorem. It is also known that the indefinite integrals of Henstock or $\mathcal{H}$ integrable multifunctions are interval multimeasures,
while   the indefinite integrals of Pettis (hence also McShane  or Birkhoff) integrable multifunctions are multimeasures.\\

\end{rem}

\begin{deff} \rm
   A multifunction $\vG:[0,1]\to cwk(X)$ is said to be {\it variationally Henstock} ({\it
McShane})
   integrable,
   if there exists an interval multimeasure  $\vPh_{\vG}: {\mathcal I} \to {cwk(X)}$ with the following property:
   for every $\ve>0$ there exists a gauge $\delta$
   on $[0,1]$ such that for each
   $\{(I_1,t_1), \dots,(I_p,t_p)\}\in\Pi_{\delta}^P$ (resp. $ \Pi_{\delta}$)
we have
\begin{eqnarray}\label{aa}
\sum_{j=1}^pd_H \left(\vPh_{\vG}(I_j),\vG(t_j)|I_j|\right)<\ve\,.
\end{eqnarray}
   We write then $$(vH)\int_0^1\vG\,dt:=\vPh_{\vG}([0,1])\qquad ((vMS)\int_0^1\vG\,dt:=\vPh_{\vG}([0,1])).$$
 The set  multifunction  $\vPh_{\vG}$  will be  called the {\it variational Henstock} ({\it McShane}) {\it primitive} of $\vG$.
\\
The variational integrals on a set $I \in \mathcal{I}$ can be defined in an analogous way and they are uniquely determined.
It has been proven in \cite[Proposition 2.8]{cdpms2016} that each variationally Henstock integrable multifunction $\vG:[0,1]\to cwk(X)$ is Bochner measurable.
   \end{deff}

Important tools for the study of multifunctions are
embeddings and variational measures.
 Let $l_{\infty}(B_{X^*})$ be the Banach space of bounded real valued functions defined on $B_{X^*}$ endowed with the supremum norm $|| \cdot||_{\infty}$.
The  R{\aa}dstr\"{o}m embedding  $i:cwk(X) \to l{}_{\infty}(B_{X^*})$, given in \cite{l1,cdpms2016} by the relation $cwk(X)\ni{W}\longrightarrow s(\cdot,W)$,
allows to consider G-integrable
multifunctions $\vG:[0,1]\to{cwk(X)}$  as G-integrable functions
 $ i \circ \vG:[0,1] \to l{}_{\infty}(B_{X^*})$.
Thanks to the embedding, a multifunction $\vG$ is G-integrable if and only if its
 image  $i\circ{G}$ in $l{}_{\infty}(B_{X^*})$ is G-integrable (G stands for any of the gauge integrals).

For what concerns the variational measure we recall that
\begin{deff}\label{vma} \rm
The \textit{variational measure $V_\Phi: \mathcal{L} \rightarrow \mathbb{R}$ } generated by an interval multimeasure $\Phi :\mathcal{I} \rightarrow cwk(X)$
is defined by
	$$V_\Phi(E):=\inf_{\delta}\left\{Var(\Phi,\delta,E):\delta\ \text{is a gauge on }E\right\},$$
where
\[Var(\Phi, \delta,E)=\sup
\left\{
\sum_{j=1}^p\|\Phi(I_j)\|\colon
\{(I_j,t_j)\}_{j=1}^p\in\Pi_{\delta}^P\;\mbox{ and } t_j \in E, j=1,\dots,p.
\right\}\]
\end{deff}

For other properties we refer to \cite{BDpM2,dpp,cdpms2016}.\\

We also  remember that for a Pettis integrable mapping $G:[0,1]\to cwk(X)$, its integral $J_G$ is a  multimeasure on the  $\sigma$-algebra $\mathcal{L}$ (cf. \cite[Theorem 4.1]{ckr2009}) that is $\lambda$-continuous. As also observed in
\cite[section 3]{ckr2009}, this means that the {\em embedded} measure $i(J_G)$ is a countably additive measure with values in  $l_{\infty}(B_{X^*})$.\\

We recall that
\begin{deff} \rm \cite[Definition 2]{nara}
A function $f:[0,1]\to{X}$ is said to be {\it Riemann measurable} on $[0,1]$ if for every $\varepsilon > 0$
there exist an $\eta> 0$ and a
closed set $F \subset [0,1]$
with $\lambda([0,1] \setminus F )<\varepsilon$
such that
$ \left\| \, \sum_{i=1}^p \left\{f(t{}_i) - f(t'_i)\right\} |I{}_i|\, \right\| <\ve\,$
whenever $\{I{}_i\}$ is a finite collection of pairwise nonoverlapping intervals with
\mbox{$\max_{1\leq i\leq p} |I{}_i| < \eta$} and $t{}_i, \ t'_i \in I{}_i \bigcap F$.
\end{deff}
According to \cite[Theorem 4]{nara} each $\mathcal H$-integrable function is Riemann measurable on $[0,1]$.
Moreover in \cite[Theorem 9]{ncm} it
was proved
 that a function $f: [0,1]\to X$ is $\mathcal M$-integrable if and only $f$ is both Riemann measurable and Pettis integrable.
So we get the following characterization, that is parallel to Fremlin's description \cite{f1994a}:
\begin{thm}\label{t11}
A function $f:[0,1]\to X$ is Birkhoff integrable if and only if it is $\mcH$-integrable and Pettis integrable.
\end{thm}

\begin{proof}
The  only if part is trivial. For the converse observe that $\mcH$-integrability implies Riemann measurability by \cite[Theorem 4]{nara}.
Moreover by \cite[Theorem 8]{f1994a} $f$ is Mc Shane integrable and
 Riemann measurability together with Mc Shane integrability
 imply $\mathcal{M}$-integrability by \cite[Theorem 7]{nara}. \qed
\end{proof}

We denote by $\mcS{}_P(\vG), \mcS{}_{MS}(\vG), \mcS{}_{\mcH}(\vG),
\mcS{}_{H}(\vG),
\mcS{}_{Bi}(\vG)=\mcS{}_{\mathcal{M}}(\vG)\ $ and
$ \mcS{}_{vH}(\vG)$,  the collections of all selections of $\vG: [0,1] \to cwk(X)$,
 which are respectively Pettis, McShane, $\mcH$, Henstock, Birkhoff  and variationally Henstock integrable.
\section{Henstock and McShane integrability of $cwk(X)$-valued multifunctions}\label{five}
\begin{prop}\label{prop1}
Let $\vG : [0, 1] \to cwk(X)$ be  such that $\vG(\cdot) \ni 0$ a.e. If
$\vG$ is  Henstock integrable  (resp. $\mcH$-integrable) on $[0,1]$,
then  it is also McShane (resp. Birkhoff, i.e. $\mathcal{M}$) integrable on $[0,1]$.
\end{prop}
\begin{proof}
 Let $i$ be the R{\aa}dstr\"{o}m embedding  of $cwk(X)$ into $l_{\infty}(B_{X^*})$.
If $\vG$ is Henstock integrable, then we just have to prove that $ i \circ \vG$ is McShane integrable.
 By the hypothesis we have that  $ i \circ \vG$ is Henstock  integrable.
Then, thanks to \cite[Corollary 9 (iii)]{f1994a}, it will be sufficient to prove
convergence in $l{}_{\infty}(B_{X^*})$ of all series of the type $\sum_n (H) \int_{I{}_n}  i \circ \vG$,
where $(I{}_n){}_n$ is any sequence of pairwise non-overlapping subintervals of $[0,1]$.\\
But $\vG$ is HKP-integrable and $s(x{}^*, \vG) \geq 0$ a.e.  for every $x{}^* \in  X{}^*$. It follows from \cite[Lemma 1]{dm2} that $\vG$ is Pettis integrable. Consequently, the range of the indefinite Pettis integral of $\vG$ via the R{\aa}dstr\"{o}m embedding is a vector measure. This fact guarantees the convergence of the series $\sum_n (H) \int_{I{}_n} i \circ \vG$, since $(P)\int_I\vG=(H)\int_I\vG$ and $i\circ ((H)\int_I\vG)=(H)\int_I i \circ \vG$, for every $I\in\mcI$.
\\
As said before, thanks to \cite[Corollary 9 (iii)]{f1994a}, $i \circ \vG$ is McShane integrable. Consequently, $\vG$ is McShane integrable.\\
\smallskip
 If $\vG$ is $\mcH$-integrable, then $ i \circ \vG$ is $\mcH$-integrable and being already McShane integrable, it is also Pettis integrable
\cite[Theorem 8]{{f1994a}}.
 Applying now Theorem \ref{t11}, we obtain Birkhoff integrability of $i\circ \vG$. This yields Birkhoff integrability of $\vG$ . \qed
\end{proof}

 Observe that from this proposition it follows that if $\vG$ is  Henstock integrable  and $\vG(\cdot) \ni 0$ a.e. then $ i \circ \vG$ is Pettis.
We remember that the relation between  Pettis integrability of $\vG$ and $i \circ \vG$
is delicate question and it is examined for example in \cite{ckr1}.

\begin{thm}\label{t4}
Let $\vG:[0,1]\to cwk(X)$ be a  multifunction. Then the following conditions are
equivalent:
\begin{enumerate}
\item[\textbf{\textit{(i)}}] $\vG$ is Henstock integrable;
\item[\textbf{\textit{(ii)}}]
${\mathcal{S}}_H(\vG)\not=\emp$ and for every
$f\in {\mathcal{S}}_H(\vG)$  the multifunction  $\vG- f$ is McShane integrable;
\item[\textbf{\textit{(iii)}}] there exists $f\in{{\mathcal{S}}_H(\vG)}$ such that  the multifunction  $G:=\vG-f$ is McShane integrable.
\end{enumerate}
\end{thm}
\begin{proof}
$\bs{(i)\Rightarrow(ii)}$ According to  \cite[Theorem 3.1]{dm} ${\mathcal{S}}_H(\vG)\not=\emptyset$.   Let $f\in \mcS_{H}(\vG)$ be fixed.
Then $\vG-f$ is also Henstock integrable (in $cwk(X)$)  and   $0\in  \vG - f$ for every $t \in  [0,1]$.  By Proposition  \ref{prop1} the multifunction $\vG - f$ is McShane integrable. Since each McShane integrable multifunction is also Henstock integrable,
 $\bs{(ii)\Rightarrow(iii)}$ is trivial,
 $\bs{(iii)\Rightarrow(i)}$  follows at once. \qed
\end{proof}
The next result generalizes \cite[Theorem 3.4]{dm}, proved there for $cwk(X)$-valued multifunctions with compact valued integrals.
\begin{thm}\label{t3}
Let $\vG:[0,1]\to {cwk(X)}$ be a  multifunction.    Then the following conditions are equivalent:
\begin{enumerate}
\item[\textbf{\textit{(i)}}]
$\vG$ is McShane integrable;
\item[\textbf{\textit{(ii)}}]
$\vG$ is  Henstock integrable  and $\mcS{}_H(\vG)\subset\mcS{}_{MS}(\vG)$.
\item[\textbf{\textit{(iii)}}]
$\vG$ is  Henstock integrable  and $\mcS{}_H(\vG)\subset\mcS{}_P(\vG)$;
\item[\textbf{\textit{(iv)}}]
$\vG$ is  Henstock integrable  and $\mcS{}_P(\vG)\not=\emp$.
\item[\textbf{\textit{(v)}}]
$\vG$ is Henstock and Pettis integrable.
\end{enumerate}
\end{thm}
\begin{proof}
  $\bs{(i)\Rightarrow(ii)}$ Pick
 $f\in \mcS{}_H(\vG)$; then, according to
Theorem \ref{t4},  $\vG=G+f$ for a McShane integrable $G$. But as $\vG$
is Pettis integrable, also $f$ is Pettis integrable (cf. \cite[Corollary
1.5]{mu4}, \cite[Corollary 2.3]{ckr2009}). In view of \cite[Theorem 8]{f1994a}, $f$ is McShane integrable.\\
 $\bs{(ii)\Rightarrow(iii)}$  is valid, because each McShane integrable
function is also Pettis integrable (\cite[Theorem 2C]{fm}).\\
$\bs{(iii)\Rightarrow(iv)}$ In view of \cite[Theorem 3.1]{dm}
$\mcS{}_H(\vG)\not=\emp$ and so \textbf{\textit{(iii)}} implies  $\mcS{}_P(\vG)\not=\emp$.
\\
 $\bs{(iv)\Rightarrow(v)}$ Take  $f\in \mcS{}_P(\vG)$.  Since $\vG$ is
Henstock integrable, it is also HKP-integrable and so applying
\cite[Theorem 2]{dm2}  we obtain a representation $\vG=G+f$, where
$G:[0,1]\to{cwk(X)}$ is Pettis integrable in $cwk(X)$.  Consequently, $\vG$
is also Pettis integrable in $cwk(X)$ and so \textbf{\textit{(v)}} holds.
\\
$\bs{(v)\Rightarrow(i)}$ In virtue of \cite[Theorem 3.1]{dm} $\vG$ has a
McShane integrable selection  $f$.    It follows from Theorem \ref{t4} that  the multifunction $G:[0,1]\to cwk(X)$
defined by $\vG(t)=G(t)+f(t)$ is McShane integrable. \qed
\end{proof}
\section{Birkhoff  and $\mcH$-integrability of $cwk(X)$-valued multifunctions}
A quick analysis of the proof of \cite[Theorem 3.1]{dm} proves the following:
\begin{prop}\label{p10}
If $\vG : [0, 1] \to cwk(X)$ is $\mcH$-integrable, then $\mcS_{\mcH}(\vG)\neq\emp$.
 If  $\vG : [0, 1] \to cwk(X)$ is Pettis and $\mcH$-integrable, then $\mcS_{Bi}(\vG)\neq\emp$.
\end{prop}
As a consequence, we have the following result:
\begin{thm}\label{t4a}
Let $\vG:[0,1]\to cwk(X)$ be a  multifunction.  Then the following conditions are
equivalent:
\begin{enumerate}
\item[\textbf{\textit{(i)}}]
$\vG$ is $\mcH$-integrable;
\item[\textbf{\textit{(ii)}}]
${\mathcal{S}}{}_{\mcH}(\vG)\not=\emp$ and for every
$f\in {\mathcal{S}}{}_{\mcH}(\vG)$  the multifunction
$\vG - f$ is Birkhoff integrable;
\item[\textbf{\textit{(iii)}}]
there exists $f\in{{\mathcal{S}}{}_{\mcH}(\vG)}$ such that  the multifunction
$\vG - f$ is Birkhoff integrable.
\end{enumerate}
\end{thm}
\begin{proof}
$\bs{(i)\Rightarrow(ii)}$ Instead of  \cite[Theorem 3.1]{dm} we apply Proposition \ref{p10}. The remaining implications are trivial. \qed
\end{proof}

Applying Theorems \ref{t4a} and \ref{t11}, we have the following:
\begin{thm}\label{t3a}
Let $\vG:[0,1]\to {cwk(X)}$ be a  multifunction.    Then the following conditions are equivalent:
\begin{enumerate}
\item[\textbf{\textit{(i)}}]
$\vG$ is Birkhoff integrable;
\item[\textbf{\textit{(ii)}}]
$\vG$ is  $\mcH$-integrable  and $\mcS{}_{\mcH}(\vG)\subset\mcS{}_{Bi}(\vG)$.
\item[\textbf{\textit{(iii)}}]
$\vG$ is  $\mcH$-integrable  and $\mcS{}_{\mcH}(\vG)\subset\mcS{}_{MS}(\vG)$.
\item[\textbf{\textit{(iv)}}]
$\vG$ is  $\mcH$-integrable  and $\mcS{}_{\mcH}(\vG)\subset\mcS{}_P(\vG)$;
\item[\textbf{\textit{(v)}}]
$\vG$ is  $\mcH$-integrable  and $\mcS{}_P(\vG)\not=\emp$.
\item[\textbf{\textit{(vi)}}]
$\vG$ is Pettis and $\mcH$-integrable.
\end{enumerate}
\end{thm}
\begin{proof}
  $\bs{(i)\Rightarrow(ii)}$ If $f\in \mcS{}_{\mcH}(\vG)$ then, according to
Theorem \ref{t4a},  $\vG=G+f$ for a Birkhoff integrable $G$. But as $\vG$
is Pettis integrable, also $f$ is Pettis integrable (cf. \cite[Corollary 2.3]{ckr2009}, \cite[Corollary1.5]{mu4}). In view of Theorem \ref{t11} $f$ is Birkhoff integrable.
\\
 $\bs{(ii)\Rightarrow(iii)\Rightarrow(iv)}$  are valid, because each Birkhoff integrable function
is McShane integrable (\cite[Proposition 4]{fun}) and each McShane integrable
function is also Pettis integrable (\cite[Theorem 2C]{fm}).
\\
$\bs{(iv)\Rightarrow(v)}$ In view of Proposition \ref{p10}
$\mcS{}_{\mcH}(\vG)\not=\emp$ and so \textbf{\textit{(iii)}} implies  $\mcS{}_P(\vG)\not=\emp$.
\\
 $\bs{(v)\Rightarrow(vi)}$ Take  $f\in \mcS{}_P(\vG)$.  Since $\vG$ is
${\mcH}$-integrable, it is also HKP-integrable and so applying
\cite[Theorem 2]{dm2}  we obtain a representation $\vG=G+f$, where
$G:[0,1]\to{cwk(X)}$ is Pettis integrable in $cwk(X)$.  Consequently, $\vG$
is also Pettis integrable in $cwk(X)$ and so \textbf{\textit{(v)}} holds.
\\
$\bs{(vi)\Rightarrow(i)}$ In virtue of Proposition \ref{p10}  $\vG$ has a
Birkhoff integrable selection  $f$.    It follows from Theorem \ref{t4a} that  the multifunction $G:[0,1]\to cwk(X)$
defined by $G:=\vG- f$ is Birkhoff integrable. \qed
\end{proof}

\section{Variationally Henstock integrable selections}
	 Now, in order to examine \cite[Question 3.11]{cdpms2016},  we are going to
 consider the existence of   variationally Henstock integrable selections for a variationally Henstock  integrable multifunction $\vG:[0,1]\to{cwk(X)}$.
In particular we extend
 \cite[Theorem 3.12]{cdpms2016} which gives only a partial answer, and we remove the hypothesis that $X$ has the Radon-Nikod\'ym property
or the hypothesis ${\mathcal{S}}_{vH} \neq \emptyset$ in the theorems of decomposition arising from the previous quoted result;
so we  give a complete answer to the open question.
\\

First of all we give the following result which extends  \cite[Theorem 3.12]{cdpms2016}.
\begin{thm}\label{T4.1}
Let  $\vG:[0,1]\to cwk(X)$ be any variationally Henstock integrable multifunction. Then  ${\mathcal{S}}_{vH} \neq \emptyset$ and every strongly
measurable selection of $\vG$ is also variationally Henstock integrable.
\end{thm}
\begin{proof}
Let us notice first that  $\vG$  is Bochner measurable and so it possesses strongly  measurable selections \cite[Proposition 3.3]{cdpms2016} (the quoted result is a consequence of \cite{hvvk}).
Let $f$ be a strongly
measurable selection of $\vG$. Then $f$ is Henstock-Kurtzweil-Pettis integrable and the mapping $G$ defined by
$G:=\vG-f$
is Pettis integrable: see \cite[Theorem 1]{dm2}.
Since $\vG$ is vH-integrable then $\vG$ is Bochner measurable (\cite[Proposition 2.8]{cdpms2016}).
As the difference of $i(\vG)$ and $i(\{f\})$, the function  $i(G)$ is strongly
measurable, together with $G$.  Therefore $G$ has essentially $d_H$-separable range (that is, there is $E\in \mathcal{L}$, with $\lambda( [0,1] \setminus E)=0$ and $G(E)$ is $d_H$-separable) and so $i(G)$ is also Pettis integrable
(see \cite[Theorem 3.4 and Lemma 3.3 and their proofs]{CASCALES2}).
\\
Now, since $\vG$ is variationally Henstock integrable,  the variational measure $V_{\Phi}$ associated to the vH-integral of $\vG$  is absolutely continuous (see  \cite[Proposition 3.3.1]{porcello}). If
$V_{\phi}$ is associated to the    Henstock-Kurzweil-Pettis integral of  $f$, then $V_{\phi}\leq V_{\Phi}$ and so it is also  absolutely continuous with respect to
$\lambda$.
 Since  $\|G\|\leq \|\vG\|+\|f\|$,  it is  clear that also $V_G$  is  $\lambda$-continuous.
\\
Then, $i(G)$ satisfies all the hypotheses of
\cite[Corollary 4.1]{BDpM2}
and therefore it is variationally Henstock integrable. But then $i(\{f\})$ is too, as the difference of $i(\vG)$ and $i(G)$, and finally $f$ is variationally Henstock integrable. \qed
\end{proof}

 \begin{rem}
\rm At this point it is worth to observe that the thesis of  Theorem \ref{T4.1}	holds true only for strongly
measurable selections of $\vG$. In general,  $\vG$ may have scalarly measurable selections which are neither strongly
measurable nor even Henstock integrable (see \cite[Proposition 3.2]{cdpms2016} and \cite[Theorem 3.7]{apr}).
\end{rem}

A decomposition result,  similar to
Theorem \ref{t4a},
can be formulated now. It is also given in \cite[Corollary 3.5]{cdpms2016a} but with a different proof.
\begin{thm}\label{decofinal}{\rm (\cite[Corollary 3.5]{cdpms2016a})}
Let $\vG:[0,1]\to cwk(X)$ be a variationally Henstock integrable
multifunction. Then $\vG$ is the sum of a variationally Henstock
integrable selection $f$ and a   Birkhoff integrable multifunction
$G:[0,1]\to cwk(X)$ that is variationally Henstock integrable.
\end{thm}
\begin{proof}
Let $f$ be any variationally Henstock integrable selection of $\vG$. Then,
as previously proved,  $\vG$ is Bochner measurable,  $f$ is strongly
measurable and the variational measures associated with their integral
functions are $\lambda$-continuous. Moreover,
$f$ is HKP-integrable and, according to \cite[Theorem 1]{dm2}, the
multifunction $G$, defined by $G:=\vG-f$, is Pettis integrable. Since $\vG$ and
$f$ are variationally Henstock integrable the same holds true for  $G$.
 Hence also $i(G)$ is variationally Henstock integrable
 and, consequently, by  \cite[Proposition 4.1]{cdpms2016},
 $G$ is also Birkhoff integrable. \qed
\end{proof}

\begin{rem}\label{r2} \rm
There is now an obvious question: Let $\vG:[0,1]\to cwk(X)$ be a variationally Henstock integrable multifunction. Does there exist a variationally Henstock integrable selection $f$ of $\vG$ such that $G:=\vG-f$ is variationally McShane integrable?

Unfortunately, in general, the answer is negative. The argument is similar to that applied in \cite{dm4}.
Assume that $X$ is separable and $g$ is the $X$-valued function constructed in \cite{dmar} that is vH (and so strongly measurable by \cite[Proposition 2.8]{cdpms2016}),  Pettis but not vMS-integrable (see \cite{dmar}). Let $\vG(t):={\rm conv}\{0,g(t)\}$. Then,  $\vG$ is vH-integrable (see \cite[Example 4.7]{cdpms2016}) but it is not vMS-integrable (\cite[Theorem 3.7]{cdpms2016} or \cite[Example 4.7]{cdpms2016}) and possesses at least one  vH-integrable selection 
 by Theorem \ref{T4.1} .
Let now $f \in {\mathcal{S}}_{vH} (\Gamma)$ and  consider the multifunction $G=\vG-f$. Clearly $G$ is vH-integrable and  $G(t)={conv}\{-f(t),g(t)-f(t)\}$ for all $t\in[0,1]$.
If we suppose that $G$ is  variationally McShane integrable, then its selections $-f, g-f$ will be Bochner integrable since they are strongly measurable and dominated by $\|G\|$, but that would mean that $g$ is Bochner integrable, contrary to the assumption.\hfill$\Box$
\end{rem}

The next  theorems \ref{t31}
extends \cite[Theorems 4.3, 4.4]{cdpms2016}. In fact we can remove the hypothesis ${\mathcal{S}}_{vH} (\Gamma) \neq \emptyset$ thanks to Theorem \ref{T4.1} and \cite[Proposition 3.6]{cdpms2016}.  Its
 proof is the same of the quoted results in \cite{cdpms2016}.
\begin{thm}\label{t31}
Let $\vG:[0,1]\to {cwk(X)}$ be a vH-integrable multifunction.
Then    the following equivalences hold true:
\begin{itemize}
\item 
$ \mcS_{vH}(\vG)\subset\mcS_{MS}(\vG)$; 
\item 
$\mcS_{vH}(\vG)\subset\mcS_P(\vG)$;
 \item
$\mcS_P(\vG)\not=\emp; $
\item
$\vG$ \mbox{is Pettis integrable;} 
\item
 $\vG$ \mbox{is McShane integrable.}
\end{itemize}
Moreover if $\vG$ is also integrably bounded, then   all the previous statements  are equivalent to  the variational McShane integrability of $\vG$.
\end{thm}

So, in particular
\begin{cor}
A function $f:[0,1]\to{X}$ is variationally McShane integrable
(= Bochner integrable, cf. {\rm \cite{dm11}})
if and only if it is variationally Henstock integrable and integrably bounded.
\end{cor}
\section{Variational $\mcH$-integral}
Recently, Naralenkov introduced  stronger forms of Henstock and McShane integrals of functions, and called them ${\mathcal H}$ and $\mathcal M$ integrals. We apply that idea to variational integrals. Since the variational McShane integral of functions coincides with Bochner integral, the same holds true for the $\mcM$-integral. In case of the variational $\mcH$-integral the situation is not as obvious, but we shall prove in this section that the variational $\mcH$-integral coincides with the variational Henstock integral.
We begin with the following strengthening of the Riemann measurability, due to \cite{nara}.
\begin{deff}\label{strongriemann}\rm
We say that a function $f:[0,1]\to X$ is {\em strongly Riemann measurable}, if for every  $\varepsilon>0$ there exist a positive number $\eta$ and a closed set $F\subset [0,1]$ such that  $\lambda([0,1]\setminus F)<\varepsilon$ and
$\sum_{k=1}^K\|f(t_k)-f(t'_k)\| \cdot|I_k|<\varepsilon$ \,
whenever $\{I_1,...,I_K\}$ is a nonoverlapping finite family of subintervals of $[0,1]$ with $\max_k|I_k|<\eta$ and,  all points $t_k,t'_k$ are chosen in $I_k\cap F$, $k=1,...,K$.
\end{deff}

\begin{lem}
If $f:[0,1]\to X$ is  strongly measurable, then $f$ is strongly Riemann measurable.
\end{lem}
\begin{proof}
Fix $\varepsilon>0$. Then there exists a closed set $F\subset [0,1]$ such that  $\lambda ([0,1]\setminus F)<\varepsilon$ and $f_{|F}$ is continuous. Since $F$ is compact, then $f_{|F}$ is uniformly continuous, and so there exists a positive number $\delta>0$ such that, as soon as $t,t'$ are chosen in $F$, with $|t-t'|<\delta$, then $\|f(t)-f(t')\|<\varepsilon$. Now, fix any finite family $\{I_1,...,I_K\}$ of non-overlapping intervals with $\max_k|I_k|<\eta$, and choose arbitrarily points $t_k,t'_k$ in $I_k\cap F$ for every $k$: then we have
$$\sum_{k=1}^K\|f(t_k)-f(t'_k)\| \cdot |I_k|< \sum_{k=1}^K\varepsilon |I_k|< \varepsilon.$$ \qed
\end{proof}
Now, in order to prove that each variationally Henstock function $f:[0,1]\to X$ is also variationally $\mcH$-integrable, we shall follow the lines of the proof of \cite[Theorem 6]{nara}, with $E=[0,1]$.\\

Another preliminary result is needed, concerning {\em interior} Perron partitions.

\begin{deff}\rm
Let $\delta:[0,1]\to \R^+$ be any gauge on $[0,1]$, and let $$P:=\{(t_1,I_1), (t_2,I_2),  \ldots,  (t_K,I_K)\} \in \Pi_{\delta}^P.$$
$P$ is said to be an {\em interior} Perron partition if $t_k\in int(I_k)$ for all $k$, except when $I_k$ contains $0$ or $1$, in which case $t_k\in int(I_k)$ or $t_k\in I_k\cap\{0,1\}$.
\end{deff}

We can observe that the result given by Naralenkov in \cite[Lemma 3]{nara}, can be expressed in the following way:

\begin{lem}{\rm  \cite[Lemma 3]{nara}} \label{ex3nara}
Let $\delta$ be a gauge on $[0,1]$, and let
$P:=\{(t_1,I_1), \ldots$,  $(t_K,I_K)\}$ be any $\delta$-fine Perron partition of $[0,1]$, where the tags $t_1,...,t_K$ are all distinct. Then, for each function $\phi:[0,1]\to{X}$ and each $\varepsilon>0$ there exists a $\delta$-fine interior Perron partition of $[0,1]$, $P':=\{(t_1,I'_1),(t_2,I'_2),...,(t_K,I'_K)\}$ such that \,
$\sum_{k=1}^K \| \phi (t_k) \|\cdot \big|\, |I_k|-|I'_k| \,\big|<\varepsilon.$
\end{lem}
Thanks to this Lemma we can obtain, for variationally Henstock integrable functions, the following result:
\begin{lem}\label{ex3}
Let $f:[0,1]\to X$ be any variationally Henstock integrable mapping, and denote by $\Phi$ its primitive, i.e. $\Phi(I)=\int_I f$, for all intervals $I$. Suppose that $\delta$ is a gauge on $[0,1]$, and  $P:=\{(t_1,I_1),(t_2,I_2),...,(t_K,I_K)\} \in \Pi_{\delta}^P$
has all the tags $t_1,...,t_K$  distinct. Then, for each $\varepsilon>0$ there exists a $\delta$-fine interior Perron partition $P':=\{(t_1,I'_1),(t_2,I'_2),...,(t_K,I'_K)\}$ of $[0,1]$, such that
\, $\sum_{k=1}^K\|f(t_k)\| \cdot \big| \, |I_k|-|I'_k|\, \big|<\varepsilon,$
and\,
$\sum_{k=1}^K \| \Phi(I_k)-\Phi(I'_k)\| \leq \varepsilon.$
\end{lem}
\begin{proof}
Since $f$ is variationally Henstock integrable, the function $t\mapsto  \Phi([0,t])$ is continuous with respect to the norm topology of $X$. \qed
\end{proof}

We are now ready to present the announced result.
\begin{thm}\label{cinquesei}
Let $\vG:[0,1]\to {cwk(X)}$ be any variationally Henstock integrable multifunction. Then it is also variationally $\mcH$-integrable.
\end{thm}
\begin{proof}
Thanks to R{\aa}dstr\"{o}m embedding Theorem we may assume that $\vG$ is a function taking values in a Banach space. Denote it by $f$. First of all, we observe that $f$ is strongly measurable, and therefore strongly Riemann measurable. Fix $\varepsilon>0$. Then there exists a sequence of pairwise disjoint closed sets $(F_n)_n$ in $[0,1]$ and a decreasing sequence $(\eta_n)_n$ in $\R^+$ tending to 0, such that the set $N:=\bigcap_n([0,1]\setminus F_n)$ has Lebesgue measure 0, and moreover such that for every integer $n$
$$\sum_{k=1}^K\bigl\|f(t_k)-f(t'_k)\bigr\| \cdot |I_k|\leq \frac{\varepsilon}{2^n}$$
holds, as soon as $(I_k)_{k=1}^K$ is any non-overlapping family of subintervals with $\max_k|I_k|<\eta_n$ and the points $t_k,t'_k$ are taken in $F_n\cap I_k$.
Now, choose any bounded gauge $\delta_0$, corresponding to $\varepsilon$ in the definition of variational Henstock integral of $f$, and set
$\delta(t)=\theta_n(t)$, when $t\in F_n$ for some index $n$, and  $\delta(t)=\delta_0$ if $t\in N$,
where
$$\theta_n(t)=\min \left\{\eta_n,\frac{1}{2}\max\{\delta_0(t),\limsup_{F_n\ni\tau \to t}\,\delta_0(\tau)\} \right\}\,.$$
$\delta$ is measurable, as proved in \cite[Theorem 6]{nara}. We shall prove now that the gauge $\delta/2$ can be chosen in correspondence with $\varepsilon$ in the notion of variational integrability of $f$.
To this aim, fix any  partition $P:=\{(t_1,I_1),...,(t_K,I_K)\} \in \Pi_{\delta/2}^P$.
Without loss of generality, we may assume that all tags $t_k$ are distinct. Indeed, if a tag $t$ is common to two intervals $I,J$ of $P$, then
$$\biggl\| f(t)|I|-\int_I f \biggr\|+\biggl\| f(t)|J|-\int_Jf\biggr\|\leq
2\max\biggl\{\,\biggl\| f(t)|I|-\int_I f \biggr\|,\biggl\| f(t)|J|-\int_Jf\biggr\| \, \biggr\}$$
and therefore the sum
$$\sum_k\biggl\|f(t_k)|I_k|-\int_{I_k} f\biggr\|$$
is dominated by twice the analogous sum evaluated on a (possibly partial) partition with distinct tags.
\\
Thanks to Lemma \ref{ex3}, there exists an {\em interior} Perron partition $P':=\{(t_k,J_k),k=1,...,K\} \in \Pi_{\delta/2}^P$
such that
\begin{equation}\label{e33}
\max \left\{ \sum_{k=1}^K\|f(t_k)\| \cdot \big||I_k|-|J_k|\big|,\quad \sum_{k=1}^K\biggl\|\int_{I_k}f-\int_{J_k}f\biggr\| \right\} \leq \varepsilon\,.
\end{equation}
 Now, we shall suitably modify the tags of $P'$; fix $k$ and consider the tag $t_k$. \\
If $t_k\in F_n$ for some $n$ and
$\limsup_{F_n\ni{s}\to{t_k}}\delta_0(s)\geq \delta_0(t_k)$, then we pick $t'_k$ in the set $int(I_k)\cap F_n$ in such a way that $\delta_0(t'_k)>\delta(t_k)$.  This is possible since then we have $\limsup_{F_n\ni{s}\to{t_k}}\delta_0(s)\geq 2\delta(t_k)$. \\
 If $t_k\in F_n$ for some $n$ and $\limsup_{F_n\ni{s}\to{t_k}}\delta_0(s)<\delta_0(t_k)$ or if $t_k\in{N}$, then we set $t_k'=t_k$.
 From this it follows that the partition $P'':=\{(t'_k,I_k):k=1,...,K\}$ is a $\delta_0$-fine  interior Perron  partition. Summarizing, we have
\begin{eqnarray*}
\sum_k\biggl\|f(t_k)|I_k|-\int_{I_k}f\biggr\| &\leq&
\sum_k\|f(t_k)\| \cdot \big||I_k|-|J_k|\big|+\sum_k\|f(t_k)-f(t'_k)\| \cdot |J_k|+\\
&+&
\sum_k\biggl\|f(t'_k)|J_k|-\int_{J_k}f\biggr\|+\sum_k\biggl\|\int_{I_k}f-\int_{J_k}f\biggr\|.
\end{eqnarray*}
Now,
\begin{eqnarray*}
\sum_k\|f(t_k)\| \cdot \big||I_k|-|J_k|\big| + \sum_k\biggl\|\int_{I_k}f-\int_{J_k}f\biggr\|\leq 2\varepsilon
\end{eqnarray*}
thanks to  (\ref{e33}), and
\begin{eqnarray*}
\sum_k\biggl\|f(t'_k)|J_k|-\int_{J_k}f\biggr\|\leq \varepsilon
\end{eqnarray*}
because $P''$ is $\delta_0$-fine.
Finally, thanks to the strong Riemann measurability,
\begin{eqnarray*}
\sum_k\|f(t_k)-f(t'_k)\|\cdot |J_k|=\sum_{t_k\in N^c}\|f(t_k)-f(t'_k)\|\cdot|J_k| \leq
\sum_n \frac{\varepsilon}{2^n}= \varepsilon,
\end{eqnarray*}
and so
$$\sum_k\biggl\|f(t_k)|I_k|-\int_{I_k}f\biggr\|\leq 4 \varepsilon$$
which concludes the proof. \qed
\end{proof}
\section*{Acknowledgments}
This is a post-peer-review, pre-copyedit version of an article published in Annali di Matematica Pura ed Applicata. The final authenticated version is available online at: http://dx.doi.org/10.1007/s10231-017-0674-z 

D. Candeloro,
	Department of Mathematics and Computer Sciences - 06123 Perugia (Italy) orcid ID: 0000-0003-0526-5334; 
     email: domenico.candeloro@unipg.it  \\

    L. Di Piazza, 
    Department of Mathematics, University of Palermo,  Via 	Archirafi 34, 90123 Palermo (Italy)  orcid ID: 0000-0002-9283-5157;				  email: luisa.dipiazza@unipa.it\\

	K. Musia{\l}, 
	Institute of Mathematics, Wroc{\l}aw University,  Pl. 			Grunwaldzki  2/4, 50-384 Wroc{\l}aw (Poland) orcid ID: 0000-0002-6443-2043;
	email: kazimierz.musial@math.uni.wroc.pl\\

	A. R. Sambucini ,
	Department of Mathematics and Computer Sciences - 06123 Perugia (Italy)  orcid ID: /0000-0003-0161-8729;
    email: anna.sambucini@unipg.it

\end{document}